\documentclass[12pt]{amsart}
\usepackage{amssymb,latexsym,amsmath,amsthm}
\usepackage{graphicx}
\newtheorem{definition}{Definition}[section]
\newtheorem{theorem}[definition]{Theorem}
\newtheorem{corol}[definition]{Corollary}

\newcommand{\F}{\mathbb{F}}
\newcommand{\fq}{\mathbb{F}_q}
\newcommand{\rmv}[1]{}

\title{Index bounds for character sums with polynomials over finite fields}
\author{Daqing Wan and Qiang Wang}
\thanks{
Research of authors was partially supported
by NSF and NSERC of Canada.}

\address{Department of Mathematics\\
University of California\\
Irvine,  CA 92697-3875}
\email{dwan@math.uci.edu}

\address{School of Mathematics and Statistics\\
Carleton University\\
1125 Colonel By Drive\\
Ottawa, ON K1S 5B6\\
Canada}
\email{wang@math.carleton.ca}



\begin{document}

\begin{abstract}
We provide an index bound for character sums of polynomials over finite fields. This improves the Weil bound for high degree polynomials with small indices, as well as polynomials with large indices that are generated by cyclotomic mappings of small indices.  As an application, we also give some general  bounds for numbers of solutions of some Artin-Schreier equations and mininum weights of some cyclic codes. 
\end{abstract}

\maketitle

\section{Introduction}

Let $g(x)$ be a polynomial of degree $n > 0$ and $\psi :\fq \rightarrow \mathbb{C}^*$ be a nontrivial additive character.  If $g(x)$ is not of the form $c+f^p-f$ for some $f(x) \in \fq[x]$ and constant $c\in \fq$, then 
\begin{equation}\label{weilbound}
\left|  \sum_{x\in \fq} \psi(g(x)) \right| \leq (n-1) \sqrt{q}. 
\end{equation}
This is the case if the degree $n$ is not divisible by $p$. 
The upper bound in Equation (\ref{weilbound}) is well known as  the Weil bound. In 1996,  Stepanov \cite{Stepanov}  stated the following problem for additive characters.

\medskip
\noindent  {\bf Problem 1.}\label{StepanovProbm} Determine the class of polynomials $g(x) \in \fq[x]$ 
of degree $n$, $1\leq n \leq q-1$  for which the upper bound (\ref{weilbound}) 
can be sharpened and the absolute value of the Weil sum can be estimated non-trivially for $n \geq \sqrt{q} +1$. 

\medskip
It is well known that every polynomial $g$ over $\mathbb{F}_q$ such that $g(0) =b$
has the form $ax^rf(x^s)+b$ with some positive integers
 $r, s$ such that $s\mid (q-1)$. There are different ways to choose $r, s$ in the form
$ax^rf(x^s)+b$.   However,  in \cite{AGW:09},  the concept of the index of a polynomial over a finite field was first introduced
and any  non-constant  polynomial $g\in\mathbb{F}_q[x]$ of
degree $n \leq q-1$  can be
written {\it uniquely} as
$g(x) = a(x^rf(x^{(q-1)/\ell}))+b$
with  index $\ell$ defined below.
Namely, write
$$g(x)=a(x^n+a_{n-i_1} x^{n-i_1}+\cdots+a_{n-i_k} x^{n-i_k})+b,$$
where $a,~a_{n-i_j}\neq 0$, $j=1, \dots, k$.  
Let $r$ be the lowest degree of $x$ in $g(x)-b$.
Then $g(x)=a\left(x^r f(x^{(q-1)/\ell}) \right)+b,$ where $f(x)=
x^{e_0}+a_{n-i_1} x^{e_1}+\cdots+ a_{n-i_{k-1}}x^{e_{k-1}} + a_{r}
$,
 $$\ell=\frac{q-1}{\gcd(n-r,n-r-i_1,\dots, n-r-i_{k-1}, q-1)} := \frac{q-1}{s},$$
and $\gcd(e_0, e_1, \dots, e_{k-1}, \ell)=1 .$ The integer $\ell=\frac{q-1}{s}$ is called the
  {\it index} of $g(x)$.  In particular, when $k=0$, we note that any polynomial $ax^r+b$ has the index $\ell =1$.  From the above definition of index $\ell$, one can see that the greatest common divisor condition makes $\ell$ minimal among those possible choices.   The index of a polynomial  is closely related to the concept of the least index of a cyclotomic mapping polynomial  \cite{Evans:92,NW:05,Wang}.
Let  $\gamma$ is a fixed primitive element of $\mathbb{F}_q$. Let $\ell \mid (q-1)$ and the set of all nonzero $\ell$-th powers be $C_0$. Then $C_0$ is a subgroup of $\mathbb{F}_q^*$ of
index $\ell$. The elements of the factor group $\mathbb{F}_q^*/C_0$
are the {\it cyclotomic cosets of index $\ell$}
$$ C_i := \gamma^i C_0, \  \  \  i = 0, 1, \cdots, \ell-1.$$

For any $a_0, a_1, \cdots, a_{\ell-1} \in \mathbb{F}_q$ and a positive integer $r$,   the  {\it $r$-th order cyclotomic mapping $f^{r}_{a_0, a_1, \cdots,
a_{\ell-1}}$ of index $\ell$ }  from $\mathbb{F}_q$ to itself  (see Niederreiter and Winterhof in \cite{NW:05} for $r=1$ or  Wang \cite{Wang}) is defined by
\begin{equation}\label{CycloMappingDef}
f^{r}_{a_0, a_1, \ldots,
a_{\ell-1}} (x) =
\left\{
\begin{array}{ll}
0, &   \mbox{if} ~ x=0; \\
a_i x^{r},   &   \mbox{if} ~x \in C_i, ~ 0\leq i \leq \ell-1. \\
\end{array}
\right.
\end{equation}

It is shown that  $r$-th order cyclotomic mappings of index $\ell$ produce the polynomials of the form  $x^r f(x^s)$ where $s =\frac{q-1}{\ell}$. Indeed,  the polynomial presentation is given by
\begin{equation*}
g(x) =\frac{1}{\ell} \sum_{j=0}^{\ell-1} \left( \sum_{i=0}^{\ell-1}  a_i \zeta^{-ji}\right) x^{js+r}, 
\end{equation*}
where $\zeta =\gamma^s$ is a fixed  primitive $\ell$-th root of unity. On the other hand, as we mentioned earlier, each polynomial $f(x)$ such that $f(0)=0$ with index $\ell$ can be written as $x^r f(x^{(q-1)/\ell})$, which is an $r$-th order cyclotomic mapping with the least index $\ell$ such that $a_i = f(\zeta^i)$ for $i=0, \ldots, \ell-1$. Obviously, the index of a polynomial can be very small for a polynomial with large degree. 

The concept of index of polynomials over finite fields appears quite useful. Recently index approach was used to study  permutation polynomials \cite{Wang2}, as well as the upper bound of value sets of polynomials over finite fields when they are not permutation polynomials \cite{MWW2}. In this paper we first provide an index bound for character sums of arbitrary polynomials. 

\begin{theorem}\label{indexThm}
Let $g(x) = x^r f(x^{(q-1)/\ell}) + b $ be any polynomial with index $\ell$. Let $\zeta$ be a primitive $\ell$-th root of unity and $n_0 = \# \{ 0 \leq i \leq \ell -1 \mid f(\zeta^i) =0 \}$. Let  $\psi :\fq \rightarrow \mathbb{C}^*$ be a nontrivial additive character.  Then 
\begin{equation}\label{indexbound}
\left|  \sum_{x\in \fq} \psi(g(x)) - \frac{q}{\ell} n_0 \right| \leq (\ell-n_0)\gcd(r, \frac{q-1}{\ell}) \sqrt{q}. 
\end{equation}
\end{theorem}

This implies that for many polynomials of large degree  with small indices (for which the Weil bound becomes trivial), we have nontrivial bound for the character sum 
in terms of indices.  

Moreover, we note that many classes of polynomials with large indices  $\ell$ (e.g., $\ell =q-1$) can be defined through cyclotomic cosets of smaller index $d$ that is also a divisor of  $q-1$. Indeed, in \cite{Wang2}, we studied a general class of polynomials of the form 
\begin{equation}\label{GenCycloMappingPoly}
g(x) = \frac{1}{d} \sum_{i=0}^{d-1} \sum_{j=0}^{d-1} \zeta^{-ji} x^{j(q-1)/d} f_i(x^{(q-1)/d}) R_i(x),
\end{equation}
where $f_i(x)$ and $R_i(x)$ are arbitrary polynomials for each $0\leq i \leq d-1$ and $\zeta$ is a primitive $d$-th root of unity.  
Here we abuse the notation and let $C_0$ be a subgroup of $\fq^*$ with index $d$ and $C_i = \gamma^i C_0$, $i=0, \ldots, d-1$ be all cyclotomic cosets of index $d$. 
Equivalently, $g$ is defined by
\begin{equation}\label{GenCycloMappingDef}
g(x) =
\left\{
\begin{array}{ll}
0, &   \mbox{if} ~ x=0; \\
a_i R_i(x),   &   \mbox{if} ~x \in C_i, ~ 0\leq i \leq d-1, \\
\end{array}
\right.
\end{equation}
where $a_i = f_i(\zeta^i)$  for $0\leq i \leq d-1$ and $\zeta$ is a primitive $d$-th root of unity.  Without loss of generality, we assume that each $R_i(x)$ is a nonzero polynomial and $f_i(x)$ can be a zero polynomial.  

More generally, we obtain

\begin{theorem}\label{main}
Let $d \mid (q-1)$ and $g(x) \in \fq[x]$ be a polynomial defined by 
\begin{equation*}
g(x) =
\left\{
\begin{array}{ll}
0, &   \mbox{if} ~ x=0; \\
a_i R_i(x),   &   \mbox{if} ~x \in C_i, ~ 0\leq i \leq d-1, \\
\end{array}
\right.
\end{equation*}
where $a_i \in \fq$, $0\neq R_i(x) \in \fq[x]$, $R_i(0)=0$, and  $C_i$ is the $i$-th cyclotomic coset of index  $d$  for $0\leq i \leq d-1$.
Let  $L= \{0 \leq i \leq d-1 \mid a_i \neq 0\}$ and $n_0 = d - | L |$. 
If the degree $r_i$ of each nonzero polynomial $R_i(x)$ satisfies that $\gcd(r_i, p)=1$ for each $i \in L$ and $r = \max \{r_i \mid i \in L \}$, then  we have 
\begin{equation}\label{indexbound}
\left|  \sum_{x\in \fq} \psi(g(x)) - \frac{q}{d} n_0 \right| \leq (d-n_0)r \sqrt{q}. 
\end{equation}

Moreover,  if $R_i(x) = x^{r_i}$ for $0\leq i\leq d-1$, then we have

\begin{equation}\label{indexbound2}
\left|  \sum_{x\in \fq} \psi(g(x)) - \frac{q}{d} n_0 \right| \leq (d-n_0)\max_{ i\in L} \{ \gcd(r_i, \frac{q-1}{d}) \} \sqrt{q}. 
\end{equation}
\end{theorem}

We note that the conditions $R_i(0) =0$ for $0\leq i \leq d-1$ in the above theorem are only used to normalize the polynomial in the proof. Moreover, a slightly looser upper bound $(d-n_0)r \sqrt{q}$ instead of $\frac{(d-n_0)(dr -1)}{d} \sqrt{q}$ is presented in the result for the sake of simplicity.  In fact, without the restrictions on the values of $R_i(x)$ at $0$, we still have the same bound as follows:

\begin{equation}\label{indexbound2}
\left|  \sum_{x\in \fq^*} \psi(g(x)) - \frac{q-1}{d} n_0 \right| \leq (d-n_0)r \sqrt{q}, 
\end{equation}
where the sum runs over all non-zero elements in $\fq$.  
Therefore we obtain nontrivial bounds  for polynomials defined by (\ref{GenCycloMappingDef}) if either each $R_i(x) = x^{r_i}$ is a suitable monomial or each $R_i(x)$ is a low degree polynomial. In Section~\ref{sec2}, we prove our main results. As a consequence,  index bounds of the number of solutions of a certain Artin-Schreier equation and minimum weights of some cyclic codes are derived in Section~\ref{sec3}.

\section{Proof of theorems and some consequences}\label{sec2}

We note that Theorem~\ref{indexThm} is a corollary of the second part of Theorem~\ref{main} when $d=\ell$ and all $r_i$'s are the same. Therefore it is enough to prove Theorem~\ref{main}. Because of the equivalence of equations (\ref{GenCycloMappingPoly}) and  (\ref{GenCycloMappingDef}), we prove  the following equivalent result.

\begin{theorem} \label{largeindexmonomial} 
Let $g(x)  = \frac{1}{d} \sum_{i=0}^{d-1} \sum_{j=0}^{d-1} \zeta^{-ji} x^{js} f_i(x^{(q-1)/d})R_i(x)$  for some $d \mid (q-1)$ and $s = \frac{q-1}{d}$ such that $R_i(0) =0$ for $1\leq i \leq d$.  Let $\zeta$ be a primitive $d$-th root of unity and $n_0 = d -|L|$ where $L= \{ 0 \leq i \leq d -1 \mid f_i(\zeta^i)  \neq 0 \}$. Let  $\psi :\fq \rightarrow \mathbb{C}^*$ be a nontrivial additive character. If the degree $r_i$ of each nonzero polynomial $R_i(x)$ satisfies that $\gcd(r_i, p)=1$ for each $i \in L$ and $r = \max \{r_i \mid i \in L \}$, then  
\begin{equation}\label{indexbound}
\left|  \sum_{x\in \fq} \psi(g(x)) - \frac{q}{d} n_0 \right| \leq (d-n_0)r \sqrt{q}. 
\end{equation}

Moreover,  if $R_i(x) = x^{r_i}$ for $0\leq i\leq d-1$, then we have

\begin{equation}\label{indexbound2}
\left|  \sum_{x\in \fq} \psi(g(x)) - \frac{q}{d} n_0 \right| \leq (d-n_0)\max_{ i\in L} \{ \gcd(r_i, \frac{q-1}{d}) \} \sqrt{q}. 
\end{equation}

\end{theorem}

\begin{proof}
We recall $\gamma$ is a fixed primitive element of $\mathbb{F}_q$ and $\zeta = \gamma^{(q-1)/d}$ be a primitive $d$-th root of unity.  Because $d\mid (q-1)$, we must have $\gcd(d, p)=1$. For $x\in C_i = \gamma^i C_0$, write $x = \gamma^i y^{d}$ for some $y\in \fq^*$ and then $g(x) = f_i(\zeta^i) R_i(\gamma^{i} y^{d}) $. Let $a_i = f_i(\zeta^i)$.  We have 

\begin{eqnarray*}
\left|  \sum_{x\in \fq} \psi(g(x)) -\frac{q}{d} n_0 \right| &=& \left| \frac{1}{d} \sum_{i=0}^{d -1} \left( 1+ \sum_{y\in \fq^*}   \psi( f_i(\zeta^i) R_i(\gamma^{i} y^{d}))  \right) -\frac{q}{d} n_0 \right| \\
&\leq&   \left| \frac{1}{d} \sum_{i \in L} \left( 1+ \sum_{y\in \fq^*}   \psi( f_i(\zeta^i) R_i(\gamma^{i} y^{d}))  \right) \right| \\
&& + \left| \frac{1}{d} \sum_{i \not\in L} \left( 1-q +\sum_{y\in \fq^*}   \psi( f_i(\zeta^i) R_i(\gamma^{i} y^{d})) \right) \right| \\
&= & \frac{1}{d}   \sum_{ i\in L}  \left| \sum_{y\in \fq}   \psi(a_iR_i(\gamma^{i} y^{d})) \right|.
\end{eqnarray*}

If all the degrees of polynomials $R_i(x)$ are less than or equal to $r$, then the Weil bound implies  Equation~(\ref{indexbound}).  Indeed,  because $\gcd(dr_i, p)=1$, we must have 
\begin{eqnarray*}
\frac{1}{d}   \sum_{ i\in L}  \left| \sum_{y\in \fq}   \psi(a_iR_i(\gamma^{i} y^{d})) \right|
&\leq & \frac{d-n_0}{d} (dr-1) \sqrt{q} \\
&\leq & (d-n_0)r \sqrt{q}.
\end{eqnarray*}

Moreover, if $R_i(x) = x^{r_i}$ for $0\leq i\leq d-1$, then   $g(x) =  f_i(\zeta^i) \gamma^{ir_i} y^{d r_i} $. Moreover, if we replace $y$ by $y^k$ such that $\gcd(d r_i, q-1) = k d r_i + b(q-1)$ and $\gcd(k, q-1)=1$ in the sum $\left|  \sum_{y\in \fq}   \psi(y^{d r_i} ) \right|$, we can reduce the degree of the monomial $y^{d r_i }$ in the sum to $\gcd(d r_i, q-1)$. Therefore, we obtain 

\begin{eqnarray*}
&& \frac{1}{d}   \sum_{ i\in L}  \left|  \sum_{y\in \fq}  \psi(a_iR_i(\gamma^{i} y^{d})) \right| \\
&\leq &   \frac{1}{d}   \sum_{ i\in L}  \left( \gcd(dr_i, q-1) -1 \right)\sqrt{q} \\
&\leq &   \frac{d-n_0}{d}   \max_{ i\in L} \{  \gcd(dr_i, q-1) - 1\} \sqrt{q} \\
&\leq &  (d-n_0)   \max_{ i\in L}  \{ \gcd(r_i, \frac{q-1}{d})\} \sqrt{q}. 
\end{eqnarray*}

\end{proof}

As a result, for any polynomial with index $\ell$ and vanishing order $r$ at $0$ such that $\gcd(r, p)=1$, if both $\ell$ and $\gcd(r, \frac{q-1}{\ell})$ are small,  we obtain a nontrivial bound for its character sum. This provides a partial answer to Problem 1 because many of these polynomials have large degrees which give the trivial Weil bound.  For example, let $g(x) = x^{2(q-1)/3 + 1} + x^{(q-1)/3+1} + x$ over $\fq$ with characteristic $p > 3$. Then the Weil bound gives the trivial result. However, we note that $g(x)$ has index $\ell = 3$, $n_0=2$,  and $r=1$. By Theorem~\ref{indexThm}, we have $\left|  \sum_{x\in \fq} \psi(g(x)) - \frac{2q}{3}  \right| \leq  \sqrt{q}$.

\begin{corol}
Let $g(x) = x^n + ax^r \in \fq[x]$ with $a \in \fq^*$ and $q-1 \geq n > r\geq 1$. 
Let  $\ell = \frac{q-1}{\gcd{(n-r, q-1)}}$,  $t = \gcd{(n, r, q-1)}$, and $u=\gcd(n-r, \ell)$.  
 Let $\psi :\fq \rightarrow \mathbb{C}^*$ be a nontrivial additive character.  
 If 
 $x^{n-r} + a$ has a solution in the subset of all $\ell$-th roots of unity of $\fq$, then
\begin{equation}
\left|  \sum_{x\in \fq} \psi(x^n+ax^r) - \frac{qu}{\ell} \right| \leq (\ell-u)t \sqrt{q}, 
\end{equation}
otherwise, 
\begin{equation}
\left|  \sum_{x\in \fq} \psi(x^n+ax^r) \right| \leq \ell t \sqrt{q}. 
\end{equation} 
\end{corol}

\begin{proof} First we note that $\gcd(r, \frac{q-1}{\ell}) = \gcd(r, \gcd{(n-r, q-1)}) =  \gcd{(n, r, q-1)} = t$. 
Let $\zeta$ be a primitive $\ell$-th root of unity and $n_0 = \# \{ 0 \leq i \leq \ell -1 \mid (\zeta^i)^{n-r} + a =0 \}$. By Theorem~\ref{indexThm} we have
 \begin{equation}
\left|  \sum_{x\in \fq} \psi(g(x)) - \frac{q}{\ell} n_0 \right| \leq (\ell-n_0)t \sqrt{q}. 
\end{equation}
Suppose $-a = \gamma^k$ for a fixed primitive element $\gamma$. If $\zeta^{i(n-r)} = \gamma^{k}$, then we have $i(n-r)s \equiv k \pmod{q-1}$ where $s = \frac{q-1}{\ell}$. This linear congruence has a solution only if $s \mid k$. In this case, it reduces to $i(n-r) \equiv k/s \pmod{\ell}$ and thus $i(n-r) \equiv k/s \pmod{\ell}$ has exactly $u=\gcd(n-r, \ell)$ solutions for $i$. Therefore, $n_0=u$ if $us \mid k$ and $n_0 =0$ otherwise. Hence we obtain either
\begin{equation}
\left|  \sum_{x\in \fq} \psi(x^n+ax^r) - \frac{qu}{\ell} \right| \leq (\ell-u)t \sqrt{q}, 
\end{equation}
or 
\begin{equation}
\left|  \sum_{x\in \fq} \psi(x^n+ax^r) \right| \leq \ell t \sqrt{q}. 
\end{equation} 
\end{proof}

We remark that  $x^{n-r} + a$ has a solution in the subset of all $\ell$-th roots of unity of $\fq$ if and only if $\frac{(q-1)u}{\ell} \mid k$ where $k = \log_{\gamma} (-a)$ is the discrete logarithm of $-a$.  Otherwise, we have the index bound $\ell t \sqrt{q}$ for binomials $x^n +a x^r$. Because
$t = \gcd(n, r, q-1)$ can easily achieve $1$, our bound for many binomials is essentially 
$\ell \sqrt{q}$. We note that if $\ell < \sqrt{q} -1$, then  $\ell < \frac{q-1}{\ell} \leq n-1$ and thus our bound $\ell \sqrt{q}$ is better than the Weil bound $(n-1)\sqrt{q}$.

\section{Some applications}\label{sec3}

In this section, we remark some applications of our index bound in counting the numbers of solutions of some algebraic curves and the minimum weights of some cyclic codes.  Let  $g \in \fq[x]$ be a polynomial and  let $N_{g,q^m}$ be the number  of solutions $(x, y) \in \F_{q^m}^2$ of an Artin-Schreier equation $y^q- y = g(x)$. Then
\begin{equation}
N_{g, q^m} = \sum_{\psi_m} \sum_{x \in \F_{q^m}} \psi_m(g(x)), 
\end{equation}
where the outer sum runs over all additive character $\psi$ of $\fq$ and 
$\psi_m(x) = \psi(Tr(x))$, and $Tr$ denotes the trace from $\F_{q^m}$ to $\fq$. 

It is well known that if $g$ has degree $n$ with $\gcd(n, q) =1$, then the Weil bound gives 
\begin{equation}\label{Artin-Schreir}
|N_{g, q^m} - q^m | \leq (n-1)(q-1) q^{m/2}. 
\end{equation}
Improving the Weil bound for the Artin-Schreier curves has received a lot of recent attentions 
because of their applications in coding theory and computer sciences, see  \cite{CX} \cite{KL} \cite{RW} 
for more details. 

As a consequence of our earlier results with the assumption that $g$ has an index $\ell$ and vanishing order $r$ at $0$ such that $\gcd(r, p)=1$, we obtain the following improvement in a different direction. 

\begin{corol}  
Let  $g \in \F_{q^m}[x]$ be a polynomial with index $\ell$ and vanishing order $r$ at $0$ such that $\gcd(r, p)=1$.  Let $n_0$ be defined as in Theorem~\ref{indexThm} and $N_{g,q^m}$ be the number  of solutions $(x, y) \in \F_{q^m}^2$ of an Artin-Schreier equation $y^q- y = g(x)$. Then

\begin{equation}\label{indexboundArtin-Schreier}
\left|  N_{g, q^m} - q^m  - \frac{(q-1)q^m}{\ell} n_0 \right| \leq (q-1) (\ell-n_0)\gcd(r, \frac{q^m-1}{\ell}) q^{m/2}. 
\end{equation}

\end{corol}

In particular, we have the following corollary.

\begin{corol}
Let $g(x) = x^n + a x^r \in \F_{q^m}[x]$ such that  $\gcd(r, p)=1$. Let $\ell = \frac{q^m-1}{\gcd(n-r, q^m-1)}$ and $t = \gcd(n, r, q^m-1)$.  Then the number of solutions $N_{g, q^m}$ of the curve $y^q - y = x^n + a x^r$ satisfies 
\begin{equation}\label{indexboundArtin-SchreierBin1}
\left|  N_{g, q^m} - q^m  \right| \leq (q-1) \ell t q^{m/2},
\end{equation}
except the case when 
$x^{n-r}+a$ has a root in the set of $\ell$-th roots of unity in $\F_{q^m}$, in which case, we have 
\begin{equation}
\left|  N_{g, q^m} - q^m - \frac{(q-1)q^m \gcd(n-r, \ell)}{\ell} \right| \leq  (q-1)(\ell-1)t q^{m/2}, 
\end{equation}
\end{corol}

We note that $x^{n-r}+a$ has a root in the set of $\ell$-th roots of unity in $\F_{q^m}$ if and only if $\frac{(q^m-1)\gcd(n-r, \ell)}{\ell} \mid k$ where $k = \log_{\gamma} (-a)$ is the discrete logarithm of $-a$.

Finally we comment on some applications on cyclic codes. Let $C$ be a cyclic code of length $N$ over $\fq$ with $\gcd(N, q)=1$. Let $\F_{q^m}$ be the splitting field of the polynomial $x^N-1$ over $\fq$ and $Tr$ be the trace function from $\F_{q^m}$ to $\fq$.  Let $\beta$ be a primitive $N$-th root of unity. 
Fix a subset $J$ of the set $\{0,1,\ldots, N-1\}$ and let $h(x)=\prod_{j\in J} m_{\beta^j}(x)$ be the generator polynomial of $C^\perp$, the orthogonal code of $C$,  where  $m_\gamma(x)$ is the minimal polynomial of $\gamma$ in $\F_{q^m}$. Then $C$ consists of the words 
\[
c_a(x)=\sum_{i=0}^{N-1} Tr(g_a(\beta^i))x^i,
\]
where $g_a(x)=\sum_{j\in J} a_j x^j$ and $a=(a_j)_{j\in J} \in (\F_{q^m})^u$ with $u= |J|$. Here $J$ is called $\beta$-check set. The weight $w(a)$ of $c_a(x)$ is given by $N-z(a)$, with $z(a)=\#\{i \mid 0\leq i \leq N-1,  Tr(g_a(\beta^i))=0 \}$.  
Let $N_1$ be the number of solutions $x\in \F_{q^m}$ of the equation  $Tr(g(x))=0$ and let $N_2$ be the number of solutions $(x,y) \in \F_{q^m}^2$ of the equation $y^q-y=g(x)$, where $g(x) \in \F_{q^m}[x]$. It is clear that $N_2=qN_1$. Using the classical Weil-Serre bound, Wolfmann \cite{Wolfmann} provided some general bounds for the mininum weights of some cyclic codes. Here we can similarly give an index bound for the minimum weights of some of these cyclic codes. 

Let $k$ be the integer such that $Nk = q^m-1$. The set of all $N$-th roots of unity over $\fq$ is also the set of $k$-powers of $\F_{q^m}^*$. Therefore $z(a)$ is the number of $x^k$ in $\F_{q^m}^*$ such that $Tr(g_a(x^k))=0$. Consider $E_k = \{ x \in \F_{q^m}^* \mid Tr(g_a(x^k)) = 0 \}$. Obviously $E_k$ is the union of $z(a)$ distinct classes modulo $G_k$, where $G_k$ is the subgroup of $\F_{q^m}^*$ of order $k$. Hence $|E_k| = kz(a) = k(N-w(a)) = q^m-1 - kw(a)$. Let $N_3$ be the number of solutions $x\in \F_{q^m}$ of the equation $Tr(g_a(x^k)) = 0$. Then $N_3 = |E_k| = q^m-1 - kw(a)$ if $Tr(g_a(0)) \neq 0$ and $N_3 = |E_k|+1 = q^m -kw(a)$ if $Tr(g_a(0)) = 0$. Combining the above discussions with Equation~(\ref{indexboundArtin-Schreier}), we obtain

\begin{corol} 
Let $\F_{q^m}$ be the splitting field of the polynomial $x^N-1$ over $\fq$ with $Nk = q^m-1$ and $\beta$ a primitive $N$-th root of unity over $\fq$. Let $C$ be a cyclic code of length $N$ over $\fq$ with 
 $J$ as $\beta$-check set. Let $\zeta$ be a primitive $\ell$-th root of unity. 
  If each nonzero member of $J$ is prime to $q$ and $g_a(x^k) = \sum_{j\in J} a_j x^{kj} = x^rf_a(x^{(q^m-1)/\ell})+b$ has index $\ell$ and vanishing order $r$ at $0$. 
  Let $n_0= \#\{0\leq i \leq \ell-1 \mid f_a(\zeta^i) =0 \}$. 

(a) If $0\in J$ and $Tr(b) =0$,  then the  weight $w(a)$ of $c_a(x)$ satisfies
\[
 \left| w(a) - \frac{q^m-q^{m-1}}{k} + \frac{(q-1)q^{m-1} n_0}{k\ell} \right| \leq \frac{(q-1)(\ell-n_0)\gcd(r, \frac{q^m-1}{\ell})}{kq} q^{m/2}.
\]

(b) If either $0 \not \in J$ or $0\in J$ and $Tr(b)\neq 0$ then the weight $w(a)$ of $c_a(x)$ satisfies
\[
 \left| w(a) - \frac{q^m-q^{m-1}-1}{k} + \frac{(q-1)q^{m-1} n_0}{k\ell} \right| \leq \frac{(q-1)(\ell-n_0)\gcd(r, \frac{q^m-1}{\ell})}{kq} q^{m/2}.
\]
\end{corol}

Therefore, if $k\cdot J= \{r, r+\frac{q^m-1}{\ell},  \ldots, r + \frac{(\ell-1)(q^m-1)}{\ell} \}$ such that $0< r< \frac{q^m-1}{\ell}$ and each member of $J$ is relatively prime to $q$, we can estimate an lower bound the minimum weight of the corresponding cyclic code. 
Because $1\leq n_0 \leq \ell-1$ for all nonzero codewords $g_a(x)$, we therefore obtain the weight of $c_a(x)$ is at least
\begin{eqnarray*}
w(a) &\geq & \frac{q^m-q^{m-1}-1}{k} - \frac{(q-1)q^{m-1} n_0}{k\ell} - \frac{(q-1)(\ell-n_0)\gcd(r, \frac{q^m-1}{\ell})}{kq} q^{m/2} \\
& \geq & \frac{q^m-q^{m-1}-1}{k} - \frac{(q-1)q^{m-1} (\ell-1) }{k\ell} - \frac{(q-1)(\ell-1)\gcd(r, \frac{q^m-1}{\ell})}{kq} q^{m/2}\\
&\geq & \frac{(q-1)q^{m-1}}{k\ell} - \frac{(q-1)(\ell-1)\gcd(r, \frac{q^m-1}{\ell})}{kq} q^{m/2} - \frac{1}{k}.
\end{eqnarray*}
Therefore the minimum weights of these cyclic codes are quite large when $m$ is large.

\end{document}